\documentclass[a4paper]{article}

\usepackage{fullpage}
\usepackage{amssymb,amsfonts,amsmath,amsthm}
\usepackage{xcolor}
\definecolor{cb2blue}{RGB}{55,126,184}
\definecolor{cb2green}{RGB}{77,175,74}
\definecolor{cb2red}{RGB}{228,26,28}
\RequirePackage{doi}
\usepackage{cite}
\usepackage{bm}
\usepackage{hyperref}
\hypersetup{colorlinks,
        citecolor=cb2green,
        linkcolor=cb2blue,
        urlcolor=cb2red}
\usepackage{enumitem}

\newcommand{\emailLink}[1]{\href{mailto:#1}{#1}}
\newcommand{\orcidLink}[1]{\href{https://orcid.org/#1}{#1}}
\newcommand{\email}[1]{\textsc{email} \emailLink{#1}}
\newcommand{\orcid}[1]{\textsc{orcid} \orcidLink{#1}}

\renewcommand{\leq}{\leqslant}
\renewcommand{\geq}{\geqslant}
\renewcommand{\succeq}{\succcurlyeq}
\renewcommand{\preceq}{\preccurlyeq}
\newcommand{\CC}{\mathcal{C}}
\renewcommand{\SS}{\mathcal{S}}
 
\newcommand{\coloneqq}{:=}
\newcommand{\R}{\mathbb{R}}

\newcommand{\n}[1]{\|#1 \|}
\newcommand{\ip}[2]{\langle #1, #2\rangle}
\newcommand{\one}{\bm{1}}
\newcommand{\drv}{\frac{\mathrm{d}}{\mathrm{d}t}}
\newcommand{\dotx}{\dot{x}}
\newcommand{\dt}{\,\mathrm{d}t}
\renewcommand{\epsilon}{\varepsilon}
\newcommand{\Dxzero}{\mathcal{D}_{x_0,\nu}}
\newcommand{\Dxzeroprox}{\Dxzero^{\mathrm{prox}}}

\DeclareMathOperator*{\argmin}{argmin}
\DeclareMathOperator{\dom}{dom}
\DeclareMathOperator{\diag}{diag}
\DeclareMathOperator{\inter}{int}
\DeclareMathOperator{\ri}{ri}
\DeclareMathOperator{\rbd}{rbd}
\DeclareMathOperator{\bd}{bd}
\DeclareMathOperator{\rank}{rank}

\newtheorem{theorem}{Theorem} 
\newtheorem{lemma}{Lemma} 
\newtheorem{corollary}{Corollary} 
\newtheorem{definition}{Definition} 
\newtheorem{assumption}{Assumption}
\newtheorem{proposition}{Proposition} 

\title{\bfseries Mirror descent algorithms with logarithmic barriers}
\author{%
	Alberto De Marchi\thanks{University of the Bundeswehr Munich, Germany. \email{alberto.demarchi@unibw.de}, \orcid{0000-0002-3545-6898}.}%
	\and%
	Yura Malitsky\thanks{Faculty of Mathematics, University of Vienna, Austria. \email{yurii.malitskyi@univie.ac.at}, \orcid{0000-0001-7325-5766}.}%
	\and%
	Adrien B. Taylor\footnote{Inria, DI ENS, CNRS, PSL Research University, France. \email{adrien.taylor@inria.fr}, \orcid{0000-0003-2509-1765}.}%
}
\date{}

\begin{document}

\maketitle

\begin{abstract}
  We derive convergence guarantees for mirror descent and proximal mirror descent with a logarithmic barrier as the distance-generating function. When the solution lies on the boundary, the Bregman divergence blows up and classical guarantees become vacuous. While existing remedies recover $O(\log k/k)$ rates by adapting classical bounds, we propose a more direct route in which the Bregman distance to the solution never appears, and which achieves the same rate with constant stepsize. We also construct an example showing that this rate cannot be improved in general. Finally, we compare the proposed approach with interior-point methods.
\end{abstract}

\section{Introduction}\label{sec:intro}

This work considers the convex minimization problem
\begin{equation}
\min_{x \in \R^d}\, f(x) \quad \text{subject to} \quad x \in \CC\cap \SS,
\label{eq:prob} 
\end{equation}
where $\CC, \SS \subset \R^d$ are nonempty closed convex sets, and $f\colon\CC\cap\SS\to \R$ is a convex function.
Here we decompose the constraint set into a \emph{simple} set $\SS$, which is treated explicitly in the optimization problem, and a \emph{complicated} set $\CC$, which is handled via the \emph{logarithmic barrier} $h\colon \CC \to \R\cup\{+\infty\}$.
We consider this function $h$ as a mirror, or distance-generating, function that defines the Bregman divergence $D_h(y,x)\coloneqq h(y)-h(x)-\ip{\nabla h(x)}{y-x}$.
With this, we can formulate the mirror descent algorithm
\begin{equation}\label{eq:MD}\tag{MD}
x_{k} = \argmin_{x\in \SS}\left\{ f(x_{k-1}) + \ip{\nabla f(x_{k-1})}{x - x_{k-1}} + \tfrac{1}{\alpha_k} D_{h}(x,x_{k-1})\right\},
\end{equation}
or its proximal, or implicit, version
\begin{equation}\label{eq:proxMD}\tag{proxMD}
x_{k} = \argmin_{x\in\SS}\left\{ f(x) + \tfrac{1}{\alpha_k} D_{h}(x,x_{k-1})\right\},
\end{equation}
where appropriate assumptions are needed to ensure that the iterates $(x_k)$ are well-defined and feasible.
We postpone the formal definitions and assumptions on $(f,h,\CC)$ until Section~\ref{sec:preliminaries}.
Mirror descent \eqref{eq:MD} was originally introduced by~\cite{nemirovskii1983problem} and later reformulated and popularized in its modern form by~\cite{beck2003mirror}.

For the sake of presentation, the only thing one needs to know about a logarithmic barrier of $\CC$ is that it is well-defined on $\inter(\CC)$ and blows up on its boundary.
Logarithmic barriers have not been particularly popular in Bregman-type methods; the exceptions only confirm our point, and we discuss this further later on.
The main reason is that the classical analysis cannot handle them.
Indeed, for either \eqref{eq:MD} or \eqref{eq:proxMD} with a fixed stepsize $\alpha_k=\alpha$ for simplicity, and under any favorable conditions on $f$, the standard sublinear convergence rate is of the type
\[
        f(x_k) - f_\star \leq \frac{D_h(x_\star, x_0)}{\alpha k},
\]
where the optimal value $f_\star$ is attained at some $x_\star$.
The issue becomes apparent since, for most problems, we expect $x_\star$ to lie on the boundary of $\CC$, in which case
\[
D_h(x_\star,x_0) = h(x_\star) - h(x_0) - \ip{\nabla h(x_0)}{x_\star-x_0} = +\infty.
\]
The aim of this paper is to resolve this issue. We derive convergence rates for both \eqref{eq:MD} and \eqref{eq:proxMD} and show their tightness.
We believe that both the technique and the tightness are of interest in their own right. Additionally, we also compare the proposed approach with  interior-point methods.

\paragraph{Relative smoothness in mirror descent.}
As mentioned above, there are some exceptions where log-barriers have been used.
One example is the relative smoothness condition~\cite{bauschke2017,lu2017,birnbaum2011distributed}.
We say that $f$ is $L$-relatively smooth with respect to $h$ if
\begin{equation}
\label{eq:rel_smooth}
D_f(x,y)\leq L D_h(x,y)\quad \forall x,y\in \dom h.
\end{equation}
Under this condition, the analysis of mirror descent~\eqref{eq:MD} with $\alpha_k = \frac{1}{L}$ largely follows the proof for gradient descent and yields the convergence bound $f(x_k) - f_\star \leq \frac{L D_h(x_\star,x_0)}{k}$, which, as already mentioned, becomes vacuous whenever $x_\star \in \bd(\CC)$.
A log-barrier is often a convenient choice of $h$ that ensures \eqref{eq:rel_smooth}.
Below, we provide two illustrative examples:
\begin{itemize}
        \item $D$-optimal design \cite[\S 2.2]{lu2017}.
        Let $H\in \R^{m\times d}$ with $d>m$ and $\rank(H)=m$, and $f(x) = -\log \det(H\diag(x)H^\top)$, $\CC = \R_{+}^d$, $\SS = \{x\colon \ip{\one}{x} = 1\}$, and $h(x)=-\sum_{i=1}^d\log x_i$. We know that $f$ is $1$-relatively smooth with respect to $h$.
        \item Poisson linear inverse problems \cite[\S 5.2]{bauschke2017}.
        Let $f(x) = D_{\phi}(b, Ax)$, $\phi(x) = \sum_ix_i\log x_i$, $\SS = \R^d$, $\CC=\R^d_{+}$, and $h(x)=-\sum_{i=1}^d\log x_i$.
        We know that $f$ is $\n{b}_1$-relatively smooth with respect to $h$.
\end{itemize}
Although these problems were among the motivations for introducing condition~\eqref{eq:rel_smooth}, until now there has been little theory providing quantitative convergence guarantees for \eqref{eq:MD} in these settings;
the notable approach of \cite{lu2017} is revisited in Section~\ref{sec:existing_approach}.

\paragraph{Interior-point methods.}
An area of optimization where log-barriers have found a stronghold, with rigorous convergence rates and complexity guarantees, is, of course, interior-point methods (IPMs); see, for instance, the classical~\cite{nesterov2018lectures} or~\cite{gondzio2025interior} for historical perspectives.
Thus, it was only natural for us to follow the same path and compare the Bregman approach with the IPM approach.
While the usual presentation of IPMs does not closely follow the Bregman setting, the latter is more convenient for us.

In fact, classical IPMs can be viewed as \emph{lazy} versions of~\eqref{eq:proxMD}, in which the iterates are computed without updating the prox-center:
\begin{equation}\label{eq:IPM}\tag{IPM}
x_{k} = \argmin_{x\in\SS}\left\{ f(x) + \tfrac{1}{\alpha_k} D_{h}(x,x_0)\right\},
\end{equation}
where $h$ is a log-barrier for the set $\CC$, while $\SS$ typically represents affine constraints, and $x_0$ is a specific point called the analytic center of $\CC\cap\SS$; see, e.g.,~\cite[Def.~5.3.3]{nesterov2018lectures}.
Classical examples include:
\begin{itemize}
        \item linear programming: $\CC=\left\{x\in\R^d:\, a_i^\top\! x\leq b_i \text{ for }i=1,\ldots,n\right\}$ with $a_1,\ldots,a_n\in\R^d$, $b_1,\ldots,b_n\in\R$, and $h(x)=-\sum_{i} \log(b_i-a_i^\top\!x)$,
        \item semidefinite programming: $\CC=\left\{x\in\R^d:\, S(x)=\sum_ix_iA_i\in\mathbb{S}_+^m\right\}$ with $\mathbb{S}_+^m$ the set of $m\times m$ positive semidefinite matrices, $A_1,\ldots,A_d\in\mathbb{S}^m$ symmetric matrices, and $h(x)=-\log\det(S(x))$.
\end{itemize}
Since \eqref{eq:IPM} can be viewed as one iteration of \eqref{eq:proxMD}, one would expect the corresponding bounds to differ as follows:
\[
f(x_k)-f_\star\leq \frac{D_h(x_\star,x_0)}{\alpha_k}
\qquad\text{whereas}\qquad
f(x_k)-f_\star\leq \frac{D_h(x_\star,x_0)}{\sum_{i=1}^{k}\alpha_i}
\]
for \eqref{eq:IPM} and \eqref{eq:proxMD}, respectively.
Here the reader should not confuse the sequences $(\alpha_k)$ in \eqref{eq:proxMD} and \eqref{eq:IPM}~--- they need not be the same.

One may therefore again be disappointed by the same issue as for mirror descent schemes: solutions typically lie on the boundary, where the guarantee explodes under classical choices of $h$.
This may also appear to be at odds with IPM-type results, where, when $h$ is a $\nu$-self-concordant barrier, one can obtain
\[
f(x_k)-f_\star\leq \frac{\nu}{\alpha_k},
\]
see \cite[Thm~5.3.10]{nesterov2018lectures}.
Perhaps counterintuitively, in Section~\ref{sec:linear} we show that \eqref{eq:proxMD} can be viewed as a direct alternative to IPMs, albeit with a slightly worse total complexity.

\paragraph{Further context.}
Related strategies that exploit the advantages of both approaches are numerous, which is natural given that IPMs and proximal-point methods are among the most iconic algorithmic tools in optimization.
Indeed, augmented Lagrangian (AL) schemes~\cite{rockafellar1976augmented} are proximal-point methods with $h=\tfrac12\|\cdot\|^2$ on dual problems, and modified barrier methods~\cite{polyak1992modified} are similar to AL (dual proximal-point) but use Bregman geometries instead; see, e.g.,~\cite{teboulle1992entropic,eckstein1993nonlinear}.
Bregman-proximal algorithms have also been used with regularized barriers and exponentially growing stepsizes~\cite{tseng1993convergence,doljansky1998interior}.
That said, to the best of our knowledge, such Bregman proximal methods were generally developed with entropy-like barriers that do not blow up at the boundary, as otherwise the complexity guarantees become void.
Finally, another motivation for this study is that combining augmented Lagrangian and interior-point techniques does yield efficient solution methods; see, e.g., the quadratic programming solvers developed in~\cite{pougkakiotis2021interior,cipolla2023proximal}.

\paragraph{Roadmap.}
The most interesting part of the paper is the convergence result (Theorem~\ref{thm:proxMD} in Section~\ref{sec:prox_bregman}) and the construction of the example showing that this rate is tight (Section~\ref{sec:tight_bounds}).
The proof for \eqref{eq:MD} mainly follows that for \eqref{eq:proxMD} but is slightly more technical (Theorem~\ref{thm:MD} in Section~\ref{sec:md}).
As the proof of Theorem~\ref{thm:proxMD} is rather tedious, it may be simpler to gain an initial understanding by reading the informal analysis of mirror flow in Section~\ref{sec:mirror_flow}.
This part can also be skipped entirely.
Finally, Section~\ref{sec:ip} presents a comparison of \eqref{eq:proxMD} with IPMs, and is entirely independent of the previous main sections.

\paragraph{Existing approach.} There is a less direct but simpler
approach in the literature that essentially gives the same rate as our
Theorem~\ref{thm:proxMD}. This approach is standard in online learning,
where log-barrier regularizers are common; see, e.g.,
\cite{cover1991universal,helmbold1998online}.
In the relative-smoothness literature, it was used in
\cite[Section 4]{lu2017} to obtain a guarantee for the $D$-optimal
design problem. Despite this, the approach is still not as well known
as it deserves to be. In Section~\ref{sec:existing_approach}, we discuss it in detail and formulate its abstract version that applies to
any algorithm satisfying a standard Bregman-divergence bound.

\paragraph{Warnings.}
Throughout the paper, we almost always remain agnostic about how the subproblems in \eqref{eq:MD} and \eqref{eq:proxMD} are solved.
The only exception is the section on IPMs, where a more refined discussion is necessary.

We use $\nabla h(x)$ to denote a chosen subgradient of $h$ at $x$, with the choice clear from the context.
For simplicity, the reader may assume that $h$ is differentiable.
However, our main result (Theorem~\ref{thm:proxMD}) uses only convexity of $f$ and $h$ and the log-barrier structure of $h$.

\section{Preliminaries}\label{sec:preliminaries}

\begin{assumption}[Basic convex optimization setting]\label{ass:basic}
        The sets \(\CC, \SS \subset\R^d\) are nonempty, closed, and convex.
        The function \(f\colon \CC \cap \SS \to\R\) is convex.
        The optimal value \(f_\star\coloneqq \min_{x\in\CC\cap \SS} f(x)\)
        is finite and is attained at some \(x_\star\in\CC\cap \SS\).
\end{assumption}
Later, when discussing \eqref{eq:MD}, we add more structural assumptions on $f$: differentiability and relative smoothness.

A nonempty convex set $\CC$ always has a nonempty relative interior, $\ri(\CC)\neq \varnothing$.%
\footnote{We use the relative interior and relative boundary only for generality; assuming instead that $\CC$ has nonempty interior would require only cosmetic changes in the paper.}
We write $\rbd(\CC)$ for the relative boundary of $\CC$, that is $\rbd(\CC) = \CC\setminus \ri(\CC)$. 

\paragraph{Logarithmic barrier.}
Suppose that the set $\CC$ is parametrized by the convex function $\phi\colon \R^d\to \R$ as
\[
\phi(x) <0~~\text{on $\ri(\CC)$\quad  and \quad $\phi(x) = 0$~ on $\rbd(\CC)$}.
\]
Then for $\nu>0$, we call $h(x) = -\nu \log(-\phi(x))$ a \emph{$\nu$-log-barrier} of the set $\CC$. 

Alternatively, one could say that $h$ is a $\nu$-log-barrier of $\CC$ if it is (i) $\nu$-exp-concave, namely the map $x\mapsto \exp(-\frac{h(x)}{\nu})$ is concave, and (ii) $h(x)\to +\infty$ whenever $x\to \rbd(\CC)$.
The exp-concavity brings a useful characterization of $h$ that is a strict improvement upon convexity of $h$.
Note that this is also one of the key properties in the definition of $\nu$-self-concordant barriers~\cite[\S 5.3]{nesterov2018lectures}.
\begin{lemma}[Characterization of exp-concavity]\label{lem:exp_concave}
        Let \(h:\CC\to\mathbb R\cup\{+\infty\}\) be closed, convex, proper, and \(\nu\)-exp-concave. Then the following statements hold:
    \begin{enumerate}[label=(\roman{*})]
        \item For any $x\in\dom h$, $y\in\CC$ and  $s_x^h\in\partial h(x)$,
        \begin{equation}\label{eq:exp-concave-ineq}
                \nu \exp\left(\frac{h(x)-h(y)}{\nu}\right)
                +\ip{ s_x^h}{y-x}
                \leq \nu.
        \end{equation}
        \item If \(\inter(\dom h)\neq \varnothing\) and \(h\) is twice continuously differentiable on it, then
        \[
                \nabla^2 h(x)
                \succeq
                \frac1\nu \nabla h(x)\nabla h(x)^\top
                \qquad\forall x\in\inter(\dom h).
        \]
        \end{enumerate}
\end{lemma}
To see that the exp-concavity is a strict improvement over convexity, it suffices to apply $1+t \leq e^t$ in \eqref{eq:exp-concave-ineq}.
The parameter $\nu$ in the definition is important: roughly speaking, it measures the complexity of the set $\CC$.
For illustration, consider the following example. Let $\CC$ be given by
\[
\CC = \{x\colon \phi_{i}(x) \leq 0 \quad \text{for $i=1,\dots, m$}\},
\]
where each $\phi_i$ is convex and differentiable.
We have at least two possibilities for introducing a log-barrier. We may consider the barrier $h(x) = -\log \left(-\max_{i\in[m]} \phi_i(x)\right)$, which is simple but \emph{nondifferentiable}.
Alternatively, we may consider
\[
h(x) = -\sum_{i=1}^m\log(-\phi_i(x)) = -m \log (-\Phi(x)),~~\text{where}\quad \Phi\colon \CC\to \R:  \Phi(x) = -\left(\prod_{i=1}^m(-\phi_i(x))\right)^{1/m}
\] 
which looks more intimidating but is \emph{differentiable}.
Here we used the fact that the geometric mean of concave functions is concave on $\CC$.

\begin{assumption}[Well-posedness of the Bregman subproblems]\label{ass:wellposed}
        For every stepsize $\alpha >0$ and every $y\in \ri(\CC)$, the Bregman subproblems associated with \eqref{eq:proxMD} and \eqref{eq:MD},
        \[
        \min_{x\in \SS}
        \left\{
        \alpha f(x)+h(x)-\ip{\nabla h(y)}{x}
      \right\}\qquad\text{and}\qquad
      \min_{x\in \SS}
        \left\{
        \ip{\alpha\nabla f(y)-\nabla h(y)}{x}+h(x)
        \right\}.
        \]
        admit minimizers.
        By construction, solutions to these problems automatically lie in $\ri(\CC)$.
\end{assumption}

Note that Assumption~\ref{ass:wellposed} can be guaranteed in several standard ways, for example by coercivity of the subproblem objectives or by compactness of the feasible set \(\CC\cap \SS\).
We keep it as an explicit assumption in order not to obscure the convergence analysis with existence issues.

Finally, the function $h$ defines the Bregman divergence
\[
        D_h(y,x) =  h(y)-h(x)-\ip{\nabla h(x)}{y-x},
        \qquad
        \forall x,y\in\ri(\CC).
\]
Similarly, we define $D_f$, which is used in Section~\ref{sec:md}.

\paragraph{A technical lemma.}
The following standard bound will be needed.
\begin{lemma}\label{lem:bound2}
        Let $B\geq A> 0$. Then 
        \[
        \left(\sqrt{B} - \sqrt{A}\right)^2\leq \frac{B-A}{4} \log\left(\frac{B}{A}\right).
        \]
\end{lemma}
\begin{proof}
By the Cauchy-Schwarz inequality, we have
\begin{align*}
        \sqrt{B} - \sqrt {A}
        =
        \frac 1 2 \int_{A}^{B} \frac{1}{\sqrt{t}} \mathrm{d}t
        \leq
        \frac 1 2\left(\int_{A}^{B} 1^2\mathrm{d}t \cdot \int_{A}^{B}\frac{1}{t} \mathrm{d}t\right)^{1/2}
        =
        \frac 1 2 \sqrt{B-A}\left(\log\Bigl(\frac{B}{A}\Bigr)\right)^{1/2},
\end{align*} 
and the desired inequality follows.
\end{proof}

\section{Proximal mirror descent}\label{sec:prox_bregman}

\subsection{New analysis}

Within this section, the following holds:
\begin{itemize}
        \item Assumption~\ref{ass:basic} (main problem is well-defined);
        \item $h$ is a $\nu$-log-barrier of $\CC$;
        \item Assumption~\ref{ass:wellposed} (prox subproblem is well-defined);
        \item $x_0\in \ri(\CC)$.
\end{itemize}

The main idea behind the convergence proof is similar to that in \cite{teboulle2023elementary}, for example, where the $O(1/k)$ rate for gradient descent, implicit or explicit, is shown for the norm of the gradient.
In the Euclidean case, this is done by accumulating terms $\n{x_k-x_{k-1}}^2 = \alpha_k^2\n{\nabla f(x_k)}^2$, whereas here we accumulate their analogues $D_h(x_{k-1},x_{k}) +D_h(x_{k},x_{k-1})$.
This, together with the exp-concavity property of log-barriers~\eqref{eq:exp-concave-ineq}, is used to handle the blow-up at the boundary.
The proof below may look somewhat synthetic because, after obtaining a proof based on the idea above, we refined it through the performance estimation framework~\cite{taylor2017smooth,dragomir2022optimal} to obtain tighter and cleaner estimates.

\begin{lemma}[Lyapunov analysis of~\eqref{eq:proxMD}]\label{lem:lyapunov_proxMD}
  Let \(B_k\geq 0\), \(x_k\in \ri(\CC)\), and let \(x_{k+1}\) be obtained by one iteration of~\eqref{eq:proxMD} with stepsize \(\alpha_{k+1}\). Set \(B_{k+1}\coloneqq B_k+\alpha_{k+1}\). Then
          \begin{equation}
        \begin{aligned}\label{eq:lyap_analysis}
                B_{k+1} \left(f(x_{k+1})-f_\star\right)& + \ip{\nabla h(x_{k+1})}{ x_{k+1}-x_\star} \leq  B_k\left(f(x_{k})-f_\star\right) + \ip{\nabla h(x_{k})}{ x_{k}-x_\star}+H_k,
        \end{aligned}
        \end{equation}
        with $H_k\leq \nu$ if $B_k=0$ and $H_k\leq \frac{\nu}{4}\log\left(\frac{B_{k+1}}{B_k}\right)$ if $B_k>0$.
\end{lemma}
\begin{proof}
        From optimality condition for $x_{k+1}$, we have that
        \begin{equation}\label{eq:12e1dq2}
                \ip{\alpha_{k+1} \nabla f(x_{k+1}) + \nabla h(x_{k+1}) - \nabla h(x_k)}{ x_{k+1}-x}\leq 0 \quad \forall x\in  \CC\cap\SS.
        \end{equation}
        In addition, convexity of $f$ yields
        \begin{equation}\label{eq:2k3ne2}
                \alpha_{k+1}  \bigl(f(x_{k+1}) - f(x)\bigr) + \ip{\nabla h(x_{k+1}) - \nabla h(x_k)}{ x_{k+1}-x}\leq 0 \quad \forall x\in \CC\cap\SS.
        \end{equation}  
        Now we add~\eqref{eq:2k3ne2} with $x = x_\star$ and $\frac{B_{k}}{\alpha_{k+1}}\times\eqref{eq:2k3ne2}$ with $x = x_k$, obtaining
        \begin{multline}\label{eq:dscjom2}
                \alpha_{k+1}\left(f(x_{k+1})-f_\star\right) + B_k\left(f(x_{k+1})-f(x_k)\right)
                                                                          + \ip{\nabla h(x_{k+1}) - \nabla h(x_k)}{ x_{k+1}-x_\star} \\
                                                                        + \frac{B_k}{\alpha_{k+1}}\ip{\nabla h(x_{k+1}) - \nabla h(x_k)}{ x_{k+1}-x_k}\leq 0.
        \end{multline}
        The two inner products with $h$ can be transformed as follows:
        \begin{multline*}
                \ip{\nabla h(x_{k+1}) - \nabla h(x_k)}{ x_{k+1}-x_\star} + \frac{B_k}{\alpha_{k+1}}\ip{\nabla h(x_{k+1}) - \nabla h(x_k)}{ x_{k+1}-x_k} \\
                =
                \ip{\nabla h(x_{k+1})}{ x_{k+1}-x_\star} -  \ip{\nabla h(x_{k})}{ x_{k}-x_\star} - \underbrace{\left(\frac{B_k}{\alpha_{k+1}}\ip{\nabla h(x_{k+1})}{ x_k - x_{k+1}} +\frac{B_{k+1}}{\alpha_{k+1}}                            \ip{\nabla h(x_k)}{ x_{k+1} - x_{k}}\right)}_{=:H_k}.
        \end{multline*}
        Substituting this back to \eqref{eq:dscjom2} gives us
        \begin{equation}\label{eq:lask2}
                B_{k+1} \left(f(x_{k+1})-f_\star\right) + \ip{\nabla h(x_{k+1})}{ x_{k+1}-x_\star}  \leq  B_k\left(f(x_{k})-f_\star\right) + \ip{\nabla h(x_{k})}{ x_{k}-x_\star} + H_k.
        \end{equation}
        The last inequality would be perfect for telescoping if we could deal with the $H_k$ terms.
        Fortunately, we can do this because of exp-concavity of $h$. Specifically,
        \begin{align*}
                \nu \exp\left(\tfrac{h(x_k)-h(x_{k+1})}{\nu}\right)+\ip{\nabla h(x_k)}{ x_{k+1}-x_k} &\leq \nu\\
                \nu \exp\left(\tfrac{h(x_{k+1})-h(x_{k})}{\nu}\right)+\ip{ \nabla h(x_{k+1})}{ x_{k}-x_{k+1}} &\leq \nu.
        \end{align*}
        Hence, summing these inequalities (with appropriate weights), we obtain
        \begin{align*}
                H_k &\leq \frac{B_k+B_{k+1}}{\alpha_{k+1}}\nu - \frac{B_k}{\alpha_{k+1}}\nu \exp\left(\tfrac{h(x_{k+1})-h(x_{k})}{\nu}\right) - \frac{B_{k+1}}{\alpha_{k+1}}\nu \exp\left(\tfrac{h(x_k)-h(x_{k+1})}{\nu}\right) \\
                &\leq \frac{\nu}{\alpha_{k+1}}\left( B_k + B_{k+1}-2\sqrt{B_{k}B_{k+1}}\right) = \frac{\nu}{\alpha_{k+1}} \left(\sqrt{B_{k+1}}-\sqrt B_k\right)^2,
        \end{align*}
        where the second inequality follows from the AM-GM inequality.
        When $B_k=0$, one has $B_{k+1}=\alpha_{k+1}$ and hence $H_k\leq \nu$.
        When $B_k>0$, $H_k$ can be bounded using Lemma~\ref{lem:bound2}:
        \begin{equation*}
                H_k\leq \frac{\nu}{\alpha_{k+1}} \left(\sqrt{B_{k+1}}-\sqrt B_k\right)^2 \leq  \frac{\nu}{4}\log\left(\frac{B_{k+1}}{B_k}\right).
                \qedhere
        \end{equation*}
\end{proof}

Lemma~\ref{lem:lyapunov_proxMD} suggests telescoping the inequality at our disposal to obtain the main convergence results for~\eqref{eq:proxMD}.
\begin{theorem}\label{thm:proxMD}
        Let $\{x_i\}_{i=0,\ldots,k}$ be the iterates of \eqref{eq:proxMD}.
        For any $k\geq 1$, it holds that
        \begin{equation}\label{eq:proxMD_guarantee}
                f(x_k)-f_\star\leq \frac{ \ip{\nabla h(x_{0})}{ x_{0}-x_\star} + 2\nu + \frac \nu 4\log\frac{A_k}{A_1}}{A_k}
        \end{equation}
        where $A_k \coloneqq \sum_{i=1}^k\alpha_i$.
\end{theorem}
\begin{proof}
        Lemma~\ref{lem:lyapunov_proxMD} suggests defining the Lyapunov function
        \[
        \Psi_k\coloneqq A_k\left(f(x_{k})-f_\star\right) + \ip{\nabla h(x_{k})}{ x_{k}-x_\star},
        \]
        and telescoping the inequality to reach $\Psi_k\leq \Psi_{k-1}+H_{k-1}\leq \ldots \leq \Psi_{0}+\sum_{i=0}^{k-1} H_{i}$. 
        On top of this, we know that $H_0\leq \nu$ (by setting $B_0=0$) and that $\sum_{i=1}^{k-1}H_i\leq \frac\nu 4\sum_{i=1}^{k-1} \log\left(\frac{A_{i+1}}{A_i}\right)= \frac\nu 4 \log\left(\frac{A_{k}}{A_1}\right)$. It remains to use exp-concavity~\eqref{eq:exp-concave-ineq} as $\ip{\nabla h(x_{k})}{x_\star- x_{k}}\leq \nu$ to deal with the second term in $\Psi_k$.
\end{proof}

One can obtain many minor variations around this base theme, for instance by explicitly accounting for specific starting points $x_0$ such that $\nabla h(x_0)=0$.
In this case, one can use $H_0=A_0=0$ and reach
\begin{equation}
    \label{eq:proxMD_guarantee_analytic_center}
    f(x_k)-f_\star\leq \frac{ \nu + \frac \nu 4\log\frac{A_k}{A_1}}{A_k},
  \end{equation}
which can more readily be compared with classical bounds for~\eqref{eq:IPM}.
Starting from the analytic center $x_0 \coloneqq \argmin_{x\in \SS} h(x)$, whose definition implies $\ip{\nabla h(x_0)}{y-x_0}\geq 0$ for all $y\in\SS$, one obtains
$\ip{\nabla h(x_0)}{x_0-x_\star}\leq 0$ and therefore
\begin{equation*}
        f(x_k)-f_\star\leq \frac{ 2\nu + \frac \nu 4\log\frac{A_k}{A_1}}{A_k}.
\end{equation*}
Similarly, one can account for starting points minimizing $B_0 f(x)+h(x)$ (with $B_0>0$) by adapting the proof with $B_k\coloneqq B_0+A_k=B_0+\sum_{i=1}^k \alpha_i$ with the convention that the empty sum is $0$.
In this case, we adapt the Lyapunov strategy to
\[ \Psi_k\coloneqq B_k\left(f(x_{k})-f_\star\right) + \ip{\nabla h(x_{k})}{ x_{k}-x_\star},\]
and telescope the inequality to reach $\Psi_k\leq \Psi_{k-1}+H_{k-1}\leq \ldots \leq \Psi_{0}+\sum_{i=0}^{k-1} H_{i}$.
Here we have to use $\sum_{i=0}^{k-1}H_i\leq \frac\nu 4\sum_{i=1}^{k-1} \log\left(\frac{B_{i+1}}{B_i}\right)= \frac\nu 4 \log\left(\frac{B_{k}}{B_0}\right)$, arriving at
\begin{align}
        B_k (f(x_k)-f_\star)-\nu&\leq B_0\left(f(x_{0})-f_\star\right) + \ip{\nabla h(x_{0})}{ x_{0}-x_\star}+\frac\nu 4 \log\left(\frac{B_{k}}{B_0}\right) \nonumber\\
        &\leq \ip{B_0 \nabla f(x_0)+\nabla h(x_{0})}{ x_{0}-x_\star}+\frac\nu 4 \log\left(\frac{B_{k}}{B_0}\right).
        \label{eq:proxMD_guarantee_other_starting_point}
\end{align}

\paragraph{Specific stepsize schedules.}
The two corollaries below concern~\eqref{eq:proxMD} with specific stepsize schedules. A first classical choice (e.g., akin to mirror descent) is that of constant stepsizes. Another possibility, more closely related to classical interior-point strategies is that of geometric stepsizes.

\begin{corollary}
        Let $\alpha_k=\alpha$.
        Let $\{x_i\}_{i=0,\ldots,k}$ be the iterates of \eqref{eq:proxMD}.
        For any $k\geq 1$, it holds that 
        \begin{equation}\label{eq:proxMD_guarantee_constantstepsize}
                f(x_k)-f_\star\leq \frac{ \ip{\nabla h(x_{0})}{ x_{0}-x_\star} + 2\nu + \frac \nu 4\log {k}}{\alpha k}.
        \end{equation}
        That is, $f(x_k)-f_\star = O(\frac{\log k}{k})$.
\end{corollary}
\begin{proof}
        In this case, $A_k=\alpha k$.
\end{proof}

\begin{corollary}\label{cor:proxMD_exponentialstepsize} 
        Let $\mu>0$ and $A_1=\alpha_1>0$, and define the sequence of stepsizes using $\alpha_{k+1}=\mu A_k$.
        Let $\{x_i\}_{i=0,\ldots,k}$ be the iterates of \eqref{eq:proxMD}.
        For any $k\geq 1$, it holds that
        \begin{equation}\label{eq:proxMD_guarantee_exponentialstepsize}
                f(x_k)-f_\star\leq \frac{ \ip{\nabla h(x_{0})}{ x_{0}-x_\star} + 2\nu + (k-1)\frac \nu 4\log(1+\mu)}{\alpha_1(1+\mu)^{k-1}}.
        \end{equation}
        It follows that $f(x_k)-f_\star=O((1+\tilde{\mu})^{-k})$ for any $\tilde{\mu}<\mu$.
\end{corollary}
\begin{proof}
        In this case, for $k>1$ we have $A_{k}=(1+\mu)A_{k-1}=(1+\mu)^{k-1} A_1 = (1+\mu)^{k-1}\alpha_1$.
\end{proof}
Similar bounds can directly be obtained by exploiting the possibility of an initialization at the analytic center $x_0$ satisfying $\ip{\nabla h(x_0)}{x_0-x_\star}\leq 0$.

\paragraph{Structured barriers.}
It is common to aim at improving the conditioning of the subproblems to be solved.
For instance, \cite{doljansky1998interior,cipolla2023proximal} add a quadratic regularization to the barriers.
One could thus wonder whether the proofs above allow for such a degree of freedom given that quadratic functions are not exp-concave in general.
For this, one can simply consider a structured distance-generating function of the form $h\coloneqq h_1+h_2$ where $h_1$ is finite on the boundary (so $D_{h_1}(x_\star,x)$ does not blow up) and $h_2$ is $\nu$-exp-concave.
In this case, the statement of Lemma~\ref{lem:lyapunov_proxMD} can be generalized to
\begin{multline*}
   B_{k+1} \left(f(x_{k+1})-f_\star\right) + \ip{\nabla h_2(x_{k+1})}{ x_{k+1}-x_\star}+ D_{h_1}(x_\star,x_{k+1})\\
   \leq  B_k\left(f(x_{k})-f_\star\right) + \ip{\nabla h_2(x_{k})}{ x_{k}-x_\star}+ D_{h_1}(x_\star,x_k)+H_k,
\end{multline*}
with the same bounds on $H_k$, thereby resulting in bounds akin to that of Theorem~\ref{thm:proxMD}
\begin{equation*}
        f(x_k)-f_\star\leq \frac{D_{h_1}(x_\star,x_0)+ \ip{\nabla h_2(x_{0})}{ x_{0}-x_\star} + 2\nu + \frac \nu 4\log\frac{A_k}{A_1}}{A_k}
\end{equation*}
and of its variations.

\subsection{Existing analysis}\label{sec:existing_approach}
As noted in the introduction, there is a simpler, albeit arguably less direct, approach. The main idea is to apply the  classical guarantee of an algorithm to a shifted point
\[
    x_\delta=(1-\delta)x_\star+\delta x_0
\]
and then use convexity of \(f\) and exp-concavity of \(h\) to obtain a
rate for \(f(x_k)-f_\star\). The next proposition formalizes this idea.

\begin{proposition}\label{prop:shift}
Let \(h\) be \(\nu\)-exp-concave on \(\CC\), let
\(x_0\in\ri(\CC)\cap\SS\), and suppose that \(x_k\in\CC\cap\SS\)
satisfies
\begin{equation}\label{eq:shifted_bound}
    f(x_k)-f(x)
    \leq
    \frac{D_h(x,x_0)}{A_k},
    \qquad
    \forall x\in\ri(\CC)\cap\SS,
\end{equation}
for some \(A_k>0\). Then
\begin{equation}\label{eq:shifted_rate}
    f(x_k)-f_\star
    \leq
    \frac{
        \ip{\nabla h(x_0)}{x_0-x_\star}
        +\nu
        +\nu\log\left(
            1+\frac{A_k(f(x_0)-f_\star)}{\nu}
        \right)
    }{A_k}.
\end{equation}
\end{proposition}

\begin{proof}
For $x_\star\in\argmin_{\CC\cap\SS}f$, define
\[
    x_\delta
    :=
    (1-\delta)x_\star+\delta x_0,
    \qquad
    \delta\in(0,1].
\]
Then \(x_\delta\in\ri(\CC)\cap\SS\). Since \(e^{-h/\nu}\) is concave
and nonnegative,
\[
    e^{-h(x_\delta)/\nu}
    \geq
    \delta e^{-h(x_0)/\nu},
    \qquad\text{and therefore}\qquad
    h(x_\delta)-h(x_0)
    \leq
    \nu\log\frac{1}{\delta}.
\]
Moreover, concavity of \(e^{-h/\nu}\) at \(x_0\) gives
\[
    \ip{\nabla h(x_0)}{x_\star-x_0}\leq \nu.
  \]
Hence
\begin{equation*}
    D_h(x_\delta,x_0)
    \leq
    \nu\log\frac{1}{\delta}
    +(1-\delta)\ip{\nabla h(x_0)}{x_0-x_\star}
    \leq
    \nu\log\frac{1}{\delta}
    +\ip{\nabla h(x_0)}{x_0-x_\star}
    +\nu\delta.
\end{equation*}
By convexity of \(f\),
\[
    f(x_\delta)-f_\star
    \leq
    \delta\bigl(f(x_0)-f_\star\bigr).
\]
Applying \eqref{eq:shifted_bound} with \(x=x_\delta\) therefore gives
\[
    f(x_k)-f_\star
    \leq
    \frac{
        \nu\log(1/\delta)
        +\ip{\nabla h(x_0)}{x_0-x_\star}
    }{A_k}
    +\delta\left(
        \frac{\nu}{A_k}+f(x_0)-f_\star
    \right).
\]
It remains to optimize over $\delta$. Choosing
\[
    \delta
    =
    \frac{\nu}{
        \nu+A_k(f(x_0)-f_\star)
    }
\]
yields \eqref{eq:shifted_rate}.
\end{proof}

Since classical guarantees of~\eqref{eq:proxMD} are precisely of the form
\eqref{eq:shifted_bound}, Proposition~\ref{prop:shift} applies directly.
In particular, the resulting rate is $ O\left(\frac{\log A_k}{A_k}\right)$, 
which is of the same order as the bound in Theorem~\ref{thm:proxMD}.

\section{Mirror descent}\label{sec:mirror_algorithms}

\subsection{Mirror descent}\label{sec:md}

Within this section, the following holds:
\begin{itemize}
        \item Assumption~\ref{ass:basic} (main problem is well-defined);
        \item $h$ is a $\nu$-log-barrier of $\CC$;
        \item Assumption~\ref{ass:wellposed} (mirror subproblem is well-defined);
        \item $x_0\in \ri(\CC)$;
        \item $f,h\colon \ri(\CC)\to \R$ are  differentiable\footnote{That is over some open sets containing $\ri(\CC)$.} and $f$ is relatively smooth with respect to $h$. 
\end{itemize}
Let us recall the now-classical assumption of \emph{relative smoothness}~\cite{bauschke2017,lu2017,birnbaum2011distributed}.

\begin{definition}[Relative smoothness]
Let \(f,h\colon\ri(\CC)\to\R\) be convex and differentiable. We say that \(f\) is
\(L\)-relatively smooth with respect to \(h\) on \(\ri(\CC)\) if
\[
        D_f(x,y)\leq L D_h(x,y),
        \qquad
        \forall x,y\in\ri(\CC).
\]
Equivalently, when \(f\) and \(h\) are twice differentiable, this means
\[
        \nabla^2 f(x)\preceq L\nabla^2h(x),
        \qquad
        \forall x\in\ri(\CC).
\]
\end{definition}

For simplicity of presentation, this section focuses on mirror descent~\eqref{eq:MD} with the fixed-stepsize strategy $\alpha_k=\alpha>0$.
The proof below combines the standard mirror descent proof with the idea used to analyze the proximal mirror algorithm.
However, it may look more convoluted, as we use one additional inequality not present in either analysis.

\begin{lemma}[Lyapunov analysis of~\eqref{eq:MD}]\label{lem:lyapunov_MD}
        Let $B_k> 0$, $x_k\in \ri(\CC)$, and let $x_{k+1}$ be obtained by one iteration of~\eqref{eq:MD} with stepsize $\alpha_{k+1}=\alpha\in \left(0,\tfrac1L\right)$.
        Set $B_{k+1}\coloneqq B_k+1$.
        Then
        \begin{equation}\label{eq:lyap_analysis_MD}   
            \alpha B_{k+1}(f(x_{k+1})-f_\star)+\ip{\nabla h(x_{k+1})}{ x_{k+1}-x_\star}
            \leq
            \alpha B_k (f(x_k)-f_\star)+\ip{\nabla h(x_k)}{ x_k-x_\star}+H_k,
        \end{equation}
        with $H_k\leq \tfrac{\nu}{B_k}$.
\end{lemma}
\begin{proof}
        From optimality condition for $x_{k+1}$, we have that
        \begin{equation}\label{eq:12e1dq3}
                \ip{\alpha \nabla f(x_{k}) + \nabla h(x_{k+1}) - \nabla h(x_k)}{ x_{k+1}-x}
                \leq
                0
                \quad \forall x\in \CC\cap\SS
        \end{equation}
        Applying relative smoothness of $f$, we get
        \[
                f(x_{k+1}) - f(x_k) - \ip{ \nabla f(x_k)}{ x_{k+1} - x_k}
                \leq
                L \left[ h(x_{k+1}) - h(x_k) - \ip{\nabla h(x_k)}{ x_{k+1}- x_k} \right]
                =
                L D_h(x_{k+1},x_k),
        \]
        which combined with \eqref{eq:12e1dq3} gives us:
        \begin{equation}\label{eq:2l3rm}
                \alpha f(x_{k+1}) - \alpha f(x_k) + \alpha\ip{ \nabla f(x_{k})}{ x_{k}-x}
                \leq
                \alpha L D_h(x_{k+1}, x_k) + \ip{ \nabla h(x_{k+1}) - \nabla h(x_k)}{ x - x_{k+1}},
        \end{equation}  
        for any $x\in  \CC\cap\SS$. Convexity of $f$ yields
        \begin{equation}\label{eq:2l3rm32dqqc}
                \alpha f(x_{k+1}) -  \alpha f(x)
                \leq
                \alpha L D_h(x_{k+1}, x_k) + \ip{ \nabla h(x_{k+1}) - \nabla h(x_k)}{ x - x_{k+1}}
                \quad \forall x\in  \CC\cap\SS.
        \end{equation}  
        Now, let us add~\eqref{eq:2l3rm32dqqc} with $x = x_\star$ and $B_k\times\eqref{eq:2l3rm32dqqc}$ with $x = x_k$.
        This yields
        \begin{multline}\label{eq:dscjom}
                \alpha\left(f(x_{k+1})-f_\star\right) +  \alpha B_k\left(f(x_{k+1})-f(x_k)\right)  - \ip{\nabla h(x_{k+1}) - \nabla h(x_k)}{ x_\star - x_{k+1}}
                \\
                - B_k\ip{\nabla h(x_{k+1}) - \nabla h(x_k)}{ x_k - x_{k+1}} - \alpha(B_k+1)L D_h(x_{k+1},x_k)
                \leq
                0,
        \end{multline}
        where the terms involving $h$ can be transformed as follows, using $0\leq \alpha\leq \tfrac1L$:
        \begin{multline}\label{eq:dscjom_2}
                \ip{\nabla h(x_{k+1}) - \nabla h(x_k)}{ x_\star - x_{k+1}} + B_k\ip{\nabla h(x_{k+1}) - \nabla h(x_k)}{ x_k - x_{k+1}} + \alpha (B_k+1) L D_h(x_{k+1},x_k)
                \\
                \leq
                \ip{\nabla h(x_{k+1})}{  x_\star - x_{k+1}} -  \ip{\nabla h(x_{k})}{  x_\star  - x_{k}}
                \\
                + \underbrace{B_k\ip{\nabla h(x_{k+1})}{ x_k - x_{k+1}} +(B_k+1)                            \ip{\nabla h(x_k)}{ x_{k+1}-x_k } +(B_k+1)D_h(x_{k+1},x_k)}_{=:H_k}
                .
        \end{multline}
        Thus, an expression for $H_k$ is $ H_k = (B_k+1)(h(x_{k+1}) - h(x_k)) + B_k\ip{\nabla h(x_{k+1})}{ x_k-x_{k+1}}$, 
        which we need to upper-bound.
        We proceed by invoking exp-concavity of $h$:
        \begin{align*}
                H_k &\leq (B_k+1)(h(x_{k+1}) - h(x_k)) +B_k\left(\nu -  \nu \exp\left(\tfrac{h(x_{k+1})-h(x_{k})}{\nu}\right)\right)
                \\
                &\leq \nu \left[(B_k+1)\log\bigl(1+\tfrac 1{B_k}\bigr) - 1\right]
                 \leq \frac{\nu}{B_k},
        \end{align*}
        where the second line is obtained by maximizing the right-hand side over $\Delta\coloneqq h(x_{k+1})-h(x_k)$ (the expression is concave in $\Delta$).
        In the third inequality, we use $\log(1+a)\leq a$ and simplify.
        The desired statement follows by combining~\eqref{eq:dscjom} with~\eqref{eq:dscjom_2} and the bound on $H_k$.
\end{proof}

We are now ready to state the main convergence results for mirror descent \eqref{eq:MD}.
\begin{theorem}\label{thm:MD}
        Let $\{x_i\}_{i=0,\ldots,k}$ be the iterates of \eqref{eq:MD} with $\alpha = \frac 1 L$.
        For any $k\geq 1$, it holds that
        \begin{equation}\label{eq:MD_guarantee}
                f(x_k)-f_\star
                \leq
                \frac{f(x_0) - f_\star + L\left( \ip{\nabla h(x_{0})}{ x_{0}-x_\star} + 2\nu +  \nu \log k\right)}{k+1}.
        \end{equation}
\end{theorem}
\begin{proof}
        Again, Lemma~\ref{lem:lyapunov_MD} suggests telescoping \eqref{eq:lyap_analysis_MD} with some $B_0>0$.
        Setting $\alpha=\tfrac1L$ and $B_0=1$ gives $B_i=i+1$, and hence
        \begin{equation*}
                (k+1)\tfrac1L\left(f(x_{k})-f_\star\right) + \ip{\nabla h(x_{k})}{ x_{k}-x_\star}
                \leq
                \tfrac1L( f(x_0)-f_\star) +  \ip{\nabla h(x_{0})}{ x_{0}-x_\star} + \nu \sum_{i=0}^{k-1}\tfrac{1}{i+1},
        \end{equation*}
        and consequently,
        \[ 
                (k+1)\tfrac1L\left(f(x_{k})-f_\star\right) -\nu
                \leq
                \tfrac1L( f(x_0)-f_\star) +  \ip{\nabla h(x_{0})}{ x_{0}-x_\star} + \nu \left(1+\log k\right).
            \qedhere
        \]
\end{proof}

\paragraph{A different and direct mirror descent result from the Bregman proximal-point perspective.}
A general MD algorithm is specified by a triple $(f,w,\alpha)$, where $f$ is the objective to be minimized, $w$ is the mirror function, and $\alpha$ is the stepsize.
In the Euclidean case, we know that GD with unit stepsize is a particular case of MD applied to the triple  $(f, \tfrac{L}{2}\n{\cdot}^2, 1)$.
In turn, it can be seen as the proximal point algorithm applied to the triple $(f, \tfrac{L}{2}\n{\cdot}^2 - f, 1)$.
Evidently, this can be generalized by replacing $\tfrac{1}{2}\n{\cdot}^2$ by generic convex $h$, as this is just an algebraic manipulation.
So, let us choose $\alpha_k=1$ and use the barrier $w= Lh-f$ in~\eqref{eq:proxMD} (the choice of using $w$ and not $h$ for the log-barrier is for notational consistency within this section), and further assume $w$ is $\nu$-exp-concave.
This yields the update
\begin{equation*}
        x_{k}
        = \argmin_{x\in\SS}\left\{ f(x) +  D_{w}(x,x_{k-1})\right\}
        = \argmin_{x\in\SS}\left\{ \ip{\nabla f(x_{k-1})}{ x} +  L D_h(x,x_{k-1})\right\},
\end{equation*}
and Theorem~\ref{thm:proxMD} directly applies with $A_k=k$:
\[
        f(x_k)-f_\star
        \leq
        \frac{\ip{L\nabla h(x_0)-\nabla f(x_0)}{ x_0-x_\star} +2\nu +\tfrac{\nu}{4}\log(k)}{k}.
\]
Of course, this assumption of $Lh-f$ being exp-concave is more constraining than that of $h$ being exp-concave, so such assumption should be motivated beyond the fact that it makes the result a direct consequence of the Bregman proximal-point analysis.

To motivate this, we simply claim that the $D$-optimal design problem~\cite[\S 2.2]{lu2017} satisfies in those assumptions.
Let $H\in \R^{m\times d}$ with $d>m$, and consider $f(x) = -\log \det(H\diag(x)H^\top)$, $\CC = \R_{+}^d$, $\SS = \{x\colon \ip{\one}{x} = 1\}$, and $h(x)=-\sum_{i=1}^d\log x_i$.
Now, let $L=1$ and $\nu=d-m>0$.
Since for any $t>0$ we have logarithmic homogeneity
\begin{equation*}
w(tx)=w(x)-\nu\log t,
\end{equation*}
differentiation at $t=1$ gives $\nabla^2 w(x)x=-\nabla w(x)$ and $x^\top\nabla^2 w(x)x=\nu$. Moreover, since $w$ is convex,  $\nabla^2w(x)\succeq0$.
Therefore, for any $v$, the Cauchy--Schwarz inequality gives
\begin{align*}
\ip{\nabla w(x)}{v}^2
=
\ip{ \nabla^2 w(x) x}{v}^2 
\leq
\ip{ \nabla^2 w(x) x}{x}
\ip{ \nabla^2 w(x) v}{v}
=
\nu \ip{ \nabla^2 w(x) v}{v}.
\end{align*}
Hence $\nabla^2w(x)\succeq \tfrac{1}{\nu}\nabla w(x)\nabla w(x)^\top$ and $w$ is $\nu$-exp-concave.

\subsection{Mirror flow}\label{sec:mirror_flow}
Since the proofs of Theorems~\ref{thm:proxMD} and~\ref{thm:MD} may look somewhat involved, it might be useful to gain intuition from the corresponding continuous-time dynamics, namely a mirror flow. We consider a simplified setting to make the exposition more transparent.

Let $\SS = \R^d$.
Suppose, in addition to $\nu$-exp-concavity, that $h\colon \CC\to\R\cup\{+\infty\}$ is twice differentiable on $\ri(\CC)$ and strictly convex. Given an initial point $x_0$, the mirror flow is defined as
\begin{equation}
  \label{eq:mf}
\drv \nabla h(x_t) = - \nabla f(x_t), \qquad t\geq 0.
\end{equation}
Alternatively, we can write it as $\nabla^2 h(x_t) \dotx_t = - \nabla f(x_t)$.

First, note that $f(x_t)$ is monotone. This is standard:
\begin{align*}
  \drv{f(x_t)} = \ip{\nabla f(x_t)}{\dotx_t} = -\ip{\nabla^2h(x_t) \dotx_t}{ \dotx_t}\leq 0.
\end{align*}
Now from \eqref{eq:mf}, we have
\begin{align*}
  \ip{\nabla f(x_t) + \drv{\nabla h(x_t)}}{ x_t - x_\star} \leq 0.
\end{align*}
By convexity of $f$ and the chain rule applied to the second term, it follows that
\begin{align}\label{eq:wkvne}
 f(x_t) - f_\star  + \drv \ip{\nabla h(x_t)}{ x_t - x_\star} - \ip{\nabla h(x_t)}{ \dotx_t} \leq 0.
\end{align}
The first two terms are exactly the same as in the Lyapunov function \eqref{eq:lyap_analysis}, and the last one is a continuous analog of the $H_k$ term, the only problematic one. This time, however, we bound it differently
\begin{align*}
  \ip{\nabla h(x_t)}{ \dotx_t} &\leq \sqrt{\ip{\nabla h(x_t)}{ \dotx_t}^2} \\
                              &= \ip{\nabla h(x_t)\nabla h(x_t)^\top \dotx_t}{ \dotx_t}^{1/2} \\
                              &\leq \sqrt {\nu} \ip{\nabla^2 h(x_t)\dotx_t}{ \dotx_t}^{1/2} & &\text{(by exp-concavity)}\\
                               &= \sqrt{\nu}\ip{-\nabla f(x_t)}{ \dotx_t}^{1/2} \\
                               &= \sqrt{\nu}\left[-\drv(f(x_t)-f_\star)\right]^{1/2}
                               .
\end{align*}
It remains only to integrate this inequality carefully. 
For brevity, let $F(t) = f(x_t)-f_\star$. Then integrating~\eqref{eq:wkvne} from $0$ to $T$ and using the latter bound for
$\ip{\nabla h(x_t)}{ \dotx_t}$ yields
\begin{align*}
  \int_0^T F(t)\dt  + \int_0^T\drv \ip{\nabla h(x_t)}{ x_t - x_\star}\dt &\leq \sqrt{\nu}\int_0^T\sqrt{-F'(t)}\dt \\
                                                                        &= \sqrt{\nu}\int_0^T\sqrt{-(t+1)F'(t)} \cdot \frac{1}{\sqrt{t+1}}\dt \\
                                                                        &\leq \left(\nu \int_T^0(t+1)F'(t) \dt \cdot \int_0^T\frac{1}{t+1}\dt\right)^{1/2} \\
                                                                        &\leq \int_T^0(t+1)F'(t) \dt + \frac{\nu}{4}\int_0^T\frac{1}{t+1}\dt \\
                                                                        &= F(0) -  (T+1)F(T) + \int_0^T F(t)\dt + \frac{\nu}{4} \log (T+1),
\end{align*}
where in the second inequality we used Cauchy-Schwarz, in the third AM-GM,
and in the last equation integration by parts.
With this, we arrive at
\[(T+1)F(T)\leq F(0) + \ip{\nabla h(x_0)}{x_0 - x_\star} + \nu +\frac{\nu}{4} \log(T+1),\]
which essentially is exactly the same bound we had in the discrete case.

\subsection{Linear objectives}\label{sec:linear}

This section illustrates that the analyses are much simpler with linear objectives, while not incurring any additional logarithmic terms.
That is, we let $\SS=\R^d$ and $c\in\R^d$ and consider the simpler problem 
\[
        \min_{x\in\CC}\,\left\{ f(x)\equiv \ip{c}{x}\right\},
\]
under the same assumptions as before:
\begin{itemize}
        \item Assumption~\ref{ass:basic} (main problem is well-defined);
        \item $h$ is a $\nu$-log-barrier of $\CC$;
        \item Assumption~\ref{ass:wellposed} (prox subproblem is well-defined);
        \item $x_0\in \ri(\CC)$.
\end{itemize}
This setup is traditional from the interior-point literature where nonlinear objectives are typically delegated to the constraint set; see, e.g.,~\cite[\S 5.3.4]{nesterov2018lectures}.
One should note that with a linear objective,~\eqref{eq:MD} and~\eqref{eq:proxMD} coincide. They also coincide with~\eqref{eq:IPM} up to a stepsize adjustment:
\[ 
        x_{k} = \argmin_{x\in\R^d}\left\{ \left(\sum_{i=1}^k\alpha_i\right)\ip{c}{x} + D_h(x,x_0)\right\}.
\]
In other words, using $\nabla h(x_k)=\nabla h(x_{k-1})-\alpha_k c=\nabla h(x_0)-\sum_{i=1}^{k}\alpha_i c$, one arrives at
\begin{multline*}
        f(x_k)-f_\star
        =
        \ip{c}{x_k-x_\star}
        =
        \frac{\ip{\nabla h(x_k)}{ x_\star-x_k}+\ip{\nabla h(x_0)}{x_k-x_0}+\ip{\nabla h(x_0)}{x_0-x_\star}}{\sum_{i=1}^{k}\alpha_i}
        \\
        \leq
        \frac{2\nu+\ip{\nabla h(x_0)}{x_0-x_\star}}{\sum_{i=1}^{k}\alpha_i},
\end{multline*}
in general, or $f(x_k)-f_\star\leq \tfrac{\nu}{\sum_{i=1}^{k}\alpha_i}$ if $\nabla h(x_0)=0$.

\paragraph{Takeaways from linear objectives.}
Overall, looking at those very simple bounds and analyses, one might think that the bounds~\eqref{eq:proxMD_guarantee} or~\eqref{eq:proxMD_guarantee_analytic_center} are just weak and overly complicated.

We prove in the next section that the additional terms incurred by updating the prox-center are actually not an artifact from our proof, and that the linear-case analysis standard of the IPM literature, although simple and elegant, does not carry much of the difficulty of the general problem.
In fact, looking back at~\eqref{eq:prob}, the linear analysis of this section is no longer generally valid when $\SS\neq \R^d$ (or, in more details, when $\SS$ is not an affine set), i.e., the proof  requires $h$ to encode all constraints.

\section{Tightness of the  bounds}\label{sec:tight_bounds}

In this section, we show tightness of the $\log(k)/k$ bound for \eqref{eq:proxMD}. For simplicity, we show this for  $\nu=1$, $\alpha_k=\alpha=1$, and $x_0$ so that $\nabla h(x_0)=0$. The more general case with $\alpha,\nu>0$ can be obtained by appropriate scaling arguments. We fix $\SS = \R^d$, while the set $\CC$ is defined below.
\begin{theorem}\label{thm:tight_prox_MD}
For any integer $k\geq 1$, there exist a function $f$ satisfying
Assumption~\ref{ass:basic} and a $1$-log-barrier $h$ for a closed convex set $\CC$ such that, when initialized at  $x_0$ for which $\nabla h(x_0)=0$, the iterates of \eqref{eq:proxMD} with $\alpha_k=\alpha=1$,  satisfy
  \begin{equation}\label{eq:tight_WC}
  k (f(x_k)-f_\star) = 1+\sum_{i=1}^{k-1}(\sqrt{i+1}-\sqrt{i})^2\sim 1 + \frac 1 4 \log k.
\end{equation}
\end{theorem}
A concrete, but tedious, construction for any fixed $k$ is provided below.
The precise values presented below originate from the idea that one can identify which inequalities used in the proof must be tightened, while the construction is closely related to previous works on performance estimation problems (PEP)~\cite{drori2014performance,taylor2017smooth,dragomir2022optimal}.
The detailed construction of the corresponding PEP is deferred to the companion note~\cite{demarchi2026mirrorpep}.
The counterexample below was initially obtained through an interactive use of an LLM, which was provided with both this PEP formulation and the convergence proofs from Section~\ref{sec:prox_bregman}; the latter indicate which inequalities in the PEP are expected to be active.

\subsection{Shape of the constructed example}\label{sec:tight_shape}

The construction relies on interpolation arguments that reconstruct convex functions from samples (see, e.g.,~\cite{taylor2017smooth}).
That is, we  provide sampled representations of the objective and barrier functions, respectively in the form,
\begin{equation}\label{eq:sampled_f_h}
        S_f=\{(x_i,s_i^f,f_i)\}_{i\in\{\star,1,2,\ldots,k\}}
        \qquad\text{and}\qquad
        S_h=\{(x_i,s_i^h,h_i)\}_{i\in\{0,1,2,\ldots,k\}},
\end{equation}
so that $f$ and $h$ respectively interpolate $S_f$ and $S_h$.
That is, $S_f$ is such that $f$ is convex with $f(x_i)=f_i$ and $s_i^f\in\partial f(x_i)$ (for all $i=\star,1,2,\ldots,k$), and $S_h$ is such that $h$ is $1$-exp-concave with $h(x_i)=h_i$ and $s_i^h\in\partial h(x_i)$ (for all $i=0,1,2,\ldots,k$).
Further, $S_f$ and $S_h$ are built so that when applying~\eqref{eq:proxMD} to the corresponding $f$ and $h$, the bound in~\eqref{eq:tight_WC} is attained.
To achieve this, $f$ and $h$ are coupled via the condition
\[s_{i+1}^h = s_{i}^h - s_{i+1}^f,\]
which encodes optimality conditions of~\eqref{eq:proxMD}, and by sharing the same $x_i$'s.
As $\SS=\R^d$, and after the interpolation procedure to recover $f$ and $h$, this last condition corresponds to $\nabla h(x_{i+1})=\nabla h(x_i)-\nabla f(x_{i+1})$ when $h$ is differentiable.

Before moving to specific sets $S_f$ and $S_h$, we describe conditions on them to ensure that one can recover $f$ and $h$ interpolating $S_f$ and $S_h$.

\paragraph{Shape of $f$ and conditions on $S_f$.}
The function $f$ is constructed in the form
\begin{equation}\label{eq:model_f}
        f(x)
        =
        \max_{i\in\{\star,1,2,\ldots,k\}}
        \bigg\{ f_i+\ip{s_i^f}{x-x_i}\bigg\},
\end{equation}
where $x_\star, x_1,\dots, x_k\in\R^d$ are given points associated with some $s_\star^f,s_1^f,\ldots,s_k^f\in\R^d$ and $f_\star,f_1,\ldots,f_k\in\R$.
This construction yields a convex function satisfying $s_i^f\in\partial f(x_i)$ and $f_i=f(x_i)$, for all $i$, if and only if
\begin{equation}\label{eq:cvx_interp}
f_j \geq f_i+\ip{s_i^f}{x_j-x_i},
        \qquad \forall i,j\in\{\star,1,2,\ldots,k\}.
\end{equation}
The proof is straightforward and can be found in, e.g.,~\cite[\S 2.3]{taylor2017smooth}.

\paragraph{Shape of $h$ and conditions on $S_h$.}
For $h$, we construct it in the form $h(x)=-\log(-\phi(x))$, with the convention $h(x)=+\infty$ when $\phi(x)\geq 0$, and with
\begin{equation}\label{eq:model_phi}
        \phi(x)
    =
    \max_{i\in\{0,1,\ldots,k\}} \left\{-\exp(-h_i)\left(1-\ip{s_i^{h}}{x-x_i}\right) \right\},
\end{equation}
where similarly $x_0,x_1,\ldots,x_k\in\R^d$, $s^{h}_0,s^{h}_1,\ldots,s^{h}_k\in\R^d$ and $h_0,h_1,\ldots,h_k\in\R$.
This time, defining $\phi_i=-\exp(-h_i)$ and $s^{\phi}_i=-\phi_i s^h_i$ ($i=0,1,\ldots,k$) the constructed~\eqref{eq:model_phi} is convex and satisfies $s_i^{\phi}\in\partial \phi(x_i)$ and $\phi_i=\phi(x_i)=-\exp(-h_i)$ for all $i$ if and only if
\begin{equation}\label{eq:phi_interp}
        \phi_j\geq \phi_i+\ip{s_i^\phi}{x_j-x_i},
        \qquad \forall i,j\in\{0,1,\ldots,k\}.
\end{equation}
Note that $t\mapsto -\log(-t)$ is increasing and convex, and hence subdifferential calculus (e.g.,~\cite[Thm~4.3.1]{hiriart2013convex}) allows us to link the subdifferentials of $\phi$ with those of $h$.
Hence the function $h=-\log(-\phi)$ is $1$-exp-concave with $h(x_i)=-\log(-\phi_i)=h_i$ and $s_i^h\in\partial h(x_i)$ (for all $i$) if and only if~\eqref{eq:phi_interp}.
Noting that $\phi_i<0$ by construction, the corresponding proof is similarly straightforward as in the previous case.

\paragraph{Shape of the domain $\CC$.}
We do not explicitly construct $\CC$.
Instead, it is implicitly constructed as $\CC = \{x\colon\phi(x)\leq 0\}$; naturally, it follows from \(\phi\) being convex that its sublevel sets are also convex and hence that $\CC$ is convex.
One can also easily verify that $h$ is a log-barrier for $\CC$, as $\dom h = \ri(\CC) = \{x\colon\phi(x)<0\}$, $h(x) =  -\log(-\phi(x))$ for $x\in\dom h$,
and \(h(x)=+\infty\) otherwise.

\subsection{Two-dimensional matching example}\label{sec:tight_details}
This section explicitly constructs the sample sets $S_f$ and $S_h$ that attain the bound in~\eqref{eq:tight_WC}. We first introduce the quantities
needed for the construction, and then verify that $S_f$ and $S_h$ satisfy the
corresponding \emph{interpolation conditions}~\eqref{eq:cvx_interp} and~\eqref{eq:phi_interp}. Set
\[
        R_k= 1+\sum_{i=1}^{k-1}(\sqrt{i+1}-\sqrt i)^2, \quad
        \gamma_i
        = 
        (\sqrt{i+1}-\sqrt i)
        \left(\tfrac1{\sqrt i}-\tfrac1{\sqrt{i+1}}\right)
        =
        \sqrt{\tfrac{i+1}{i}}+\sqrt{\tfrac{i}{i+1}}-2 \qquad i=1,\ldots,k-1.
\]
Then define the function values $\{f_i\}_{i\in\{\star,1,2,\ldots,k\}}$
\begin{equation}
  \label{eq:constr_f}
f_\star = 0, \qquad 
        f_k= \tfrac{R_k}{k},
        \qquad
        f_i= f_{i+1}+\gamma_i,
        \qquad i=k-1, k-2,\ldots,1,  
\end{equation}
Finally, define a sequence $\{d_i\}_{i\in\{0, 1,\ldots,k\}}$ as
\begin{equation}\label{eq:d_i}
d_k= -1,
\qquad \text{and}\qquad
d_{i-1}= d_i+f_i+\sqrt{\tfrac{i-1}{i}}-1, \qquad i=k,k-1,\ldots,1.
\end{equation}

\paragraph{Iterates and gradients.}
We work in \(\R^2\). 
Set $x_\star= 0$, $s_\star^f= 0$, $s_0^h= 0$, and define for $i=1,\ldots,k$,
\begin{equation}\label{eq:x_i_and_s_i}
        x_i= 
        \begin{pmatrix}
        1\\
        1/\sqrt i
        \end{pmatrix},\quad
        s_i^h= 
        \begin{pmatrix}
        d_i+1\\
        -\sqrt i
        \end{pmatrix},\quad
        s_i^f= 
        \begin{pmatrix}
        d_{i-1}-d_i\\
        \sqrt i-\sqrt{i-1}
        \end{pmatrix},
\end{equation}
where one can verify $s_i^f=s_{i-1}^h-s_i^h$ for all $i=1,\ldots,k$.
The essential part of this construction is the second coordinate of
$x_i, s^f_i$, and $s^h_i$. The first one only ensures that 
$\ip{s_i^f}{x_i} = f_i$, as shown below.

\paragraph{Verification of the convex interpolation conditions for $f$.}
Now, we verify that~\eqref{eq:cvx_interp} holds.
Note that, for any $i,j\in\{1,2,\ldots,k\}$, the first coordinate of $s_i^f$ is irrelevant because $x_j-x_i$ has a zero first coordinate.

\begin{enumerate}[label=($f$.\roman{*})]
        \item (\(i=\star\)).
        Since \(s_\star^f=0\) and \(f_\star=0\), the interpolation inequalities reduce to
        $f_j\geq0$, $j=1,\ldots,k$,
        which holds by construction.
        \item (\(j=\star\)).
        We need to verify
        $f_i+\ip{s_i^f}{x_\star-x_i}\leq f_\star=0$. This is true due to  \(\ip{s_i^f}{x_i}=f_i\), which is easy to check
        \begin{align*}
          \ip{s_i^f}{x_i} = d_{i-1} - d_i + 1 - \sqrt{\frac{i-1}{i}} = f_i.
        \end{align*}
        
        \item (All remaining cases).
        Let \(i\geq 1\).
        We use
        $\ip{s_i^f}{x_j-x_i}
                =
                (\sqrt i-\sqrt{i-1})
                \left(
                \tfrac1{\sqrt j}-\tfrac1{\sqrt i}
                \right)$.
        If \(j<i\), then
        \[
        \begin{aligned}
                f_j-f_i
                &=
                \sum_{\ell=j}^{i-1}\gamma_\ell                                     =
                \sum_{\ell=j}^{i-1}
                (\sqrt{\ell+1}-\sqrt \ell)
                \left(
                \tfrac1{\sqrt \ell}
                -
                \tfrac1{\sqrt{\ell+1}}
                \right)                                                            \\
                &\geq
                (\sqrt i-\sqrt{i-1})
                \sum_{\ell=j}^{i-1}
                \left(
                \tfrac1{\sqrt \ell}
                -
                \tfrac1{\sqrt{\ell+1}}
                \right)                                                           =
                (\sqrt i-\sqrt{i-1})
                \left(
                \tfrac1{\sqrt j}-\tfrac1{\sqrt i}
                \right)                                                           =
                \ip{s_i^f}{x_j-x_i},
        \end{aligned}
        \]
        where we used that \(\ell\mapsto \sqrt{\ell+1}-\sqrt \ell\) is
        decreasing.
        If \(j>i\), then similarly
        \[
                f_i-f_j
                =
                \sum_{\ell=i}^{j-1}\gamma_\ell                                      \leq
                (\sqrt i-\sqrt{i-1})
                \sum_{\ell=i}^{j-1}
                \left(
                \tfrac1{\sqrt \ell}
                -
                \tfrac1{\sqrt{\ell+1}}
                \right)                                                           =
                (\sqrt i-\sqrt{i-1})
                \left(
                \tfrac1{\sqrt i}-\tfrac1{\sqrt j}
                \right)                                                           =
                -\ip{s_i^f}{x_j-x_i}.
        \]
        Thus the interpolation inequalities hold in all remaining cases.
\end{enumerate}

\paragraph{Verification of the exp-concave interpolation conditions for $h$.}
Now, we verify that~\eqref{eq:phi_interp} holds.
Recall \eqref{eq:x_i_and_s_i} and further define
\[
        \phi_0=  -1,
        \qquad
        h_0= -\log(-\phi_0)=0,
        \qquad
        s_0^h= 0,
        \qquad
        s_0^\phi= 0
\]
and, for $i=1,\ldots,k$,
\[
        \phi_i= -\tfrac{c}{\sqrt i},
        \qquad
        h_i= -\log(-\phi_i),
        \qquad
        s_i^\phi= -\phi_i s_i^h,
\]
where \(0<c<1\).
Finally, choose
\[
        x_0=
        \begin{pmatrix}
        1\\
        1/c
        \end{pmatrix}.
      \]
We define $x_0$ only now, because for \eqref{eq:proxMD} interpolation of $f$ does not require $x_0$.

\begin{enumerate}[label=($\phi$.\roman{*})]
        \item (\(i=0\)).
        Since \(s_0^\phi=0\), the inequalities reduce to
        $\phi_j\geq\phi_0=-1$, $j=0,1,\ldots,k$.
        This is immediate for \(j=0\).
        For \(j=1,\ldots,k\), it follows from $\phi_j=-\tfrac{c}{\sqrt j}\geq -1$ because \(0<c<1\).

        \item (\(j=0\)).
                For \(i=1,\ldots,k\), using \(s_i^\phi=-\phi_i s_i^h\), $x_0-x_i =         \begin{pmatrix}
        0\\
        1/c - 1/\sqrt i
        \end{pmatrix}$, 
        we get
        \begin{align*}
                        \phi_i+\ip{s_i^\phi}{x_0-x_i} &=           
                \phi_i
                \left(
                1-\ip{s_i^h}{x_0-x_i}
                \right)
                =
                -\tfrac{c}{\sqrt i}
                \left(
                1 + \sqrt{i}\left(\tfrac{1}{c} - \tfrac{1}{\sqrt i}\right) 
                                \right)
               = -1 = \phi_0,      
        \end{align*}
      so that  the interpolation inequalities~\eqref{eq:phi_interp} are tight when $j=0$.
        
              \item (\(i,j\in\{1,\ldots,k\}\)).
                Using that \(s_i^\phi=-\phi_i s_i^h\), $x_j-x_i =         \begin{pmatrix}
        0\\
        1/\sqrt j - 1/\sqrt i
        \end{pmatrix}$, 
                we have
                \begin{align*}
                  \phi_i+\ip{s_i^\phi}{x_j-x_i} &= \phi_i \left(1 -\ip{s_i^h}{x_j-x_i}\right) = \phi_i \left(1 +\sqrt{i}\left(\tfrac{1}{\sqrt{j}} - \tfrac{1}{\sqrt i} \right)\right) = \phi_i \sqrt{\tfrac{i}{j}} = \phi_j,
                \end{align*}
                so that the interpolation inequalities~\eqref{eq:phi_interp} are again tight for all $i,j\in\{1,\ldots,k\}$.
\end{enumerate}

\paragraph{Domain verification.}
Finally, we need to show that $\phi(x_\star)\geq 0$, so that $x_\star\in\CC$.
Given the construction~\eqref{eq:model_phi}, it suffices to verify that one term in the max is zero, which follows from $1-\ip{s_k^{h}}{x_\star-x_k}=0$.

\paragraph{Conclusion.}
We have shown that the data $\{(x_i,s_i^f,f_i)\}_{i\in\{\star,1,\ldots,k\}}$ is correctly interpolated by the convex function~\eqref{eq:model_f}, and that the data $\{(x_i,s_i^h,h_i)\}_{i\in\{0,1,\ldots,k\}}$ is correctly interpolated by the $1$-exp-concave function $-\log(-\phi)$, with the convex function $\phi$ from~\eqref{eq:model_phi}.
As this two-dimensional construction also satisfies optimality conditions of~\eqref{eq:proxMD}, it does correspond to a possible run of~\eqref{eq:proxMD} on the problem $\min_{x\in \CC}f(x)$ with barrier $h$.
Finally, as this example achieves~\eqref{eq:tight_WC}, we have proved Theorem~\ref{thm:tight_prox_MD}.

\section{Proximal mirror descent vs interior-point methods}\label{sec:ip}

In this section, we always assume that $h$ is a $\nu$-log-barrier.

In the mirror-descent setup, it is commonly assumed that the subproblems in \eqref{eq:MD} or \eqref{eq:proxMD} are simple enough to be solved in closed form or by inexpensive numerical procedures, such as a one-dimensional Newton method; see, for example, \cite{bauschke2017,lu2017}. In this case, the number of mirror-descent iterations is representative of the total computational cost. This is the \emph{first-order point of view}~--- the method has a relatively slow convergence rate but a cheap iteration.

In contrast, IPMs have a fast convergence rate but an expensive iteration that typically requires an inner optimization routine, such as Newton's method. In this case, the computational complexity is typically measured in terms of the total number of Newton iterations. This is the \emph{second-order point of view}. Interestingly, both points of view apply to \eqref{eq:proxMD}, and this section presents the latter.

As we mentioned in the introduction, classical IPMs can be viewed as lazy versions of \eqref{eq:proxMD}
\begin{equation}\tag{IPM}
x_{k} = \argmin_{x\in\SS}\left\{ f(x) + \tfrac{1}{\alpha_k} D_{h}(x,x_0)\right\},
\end{equation}
which enjoys the guarantee (e.g.,~\cite[Thm~5.3.10]{nesterov2018lectures} or~\eqref{eq:proxMD_guarantee_analytic_center} with $k=1$) when $x_0$ is the analytic center:
\[f(x_k)-f_\star \leq \frac{\nu}{\alpha_k},\]
which is fast when $\alpha_k$ grows geometrically $\alpha_k = \alpha_{k-1}(1+\mu)$ for $\mu>0$.
On the other hand, we  know that \eqref{eq:proxMD} 
\begin{equation}\tag{proxMD}
x_{k} = \argmin_{x\in\SS}\left\{ f(x) + \tfrac{1}{\alpha_k} D_{h}(x,x_{k-1})\right\},
\end{equation}
can also be fast under quite broad assumptions. In particular,
Corollary~\ref{cor:proxMD_exponentialstepsize} states that, for geometrically growing
stepsizes $\alpha_k$, one has, crudely speaking and for a fixed $x_0$:
\begin{equation*}
        f(x_k)-f_\star \leq O\left(\frac{k \nu}{(1+\mu)^k}\right).
\end{equation*}
Clearly, individual iterations of \eqref{eq:IPM} and \eqref{eq:proxMD} are computationally equivalent (up to initialization), so it makes sense to conduct a thorough  comparison. 

We therefore have the following contest: an IPM with rigid assumptions and well-understood complexity versus a newcomer~--- \eqref{eq:proxMD}~--- which is much more general but also less understood from a complexity perspective. Can the plain proximal method stand up to the established IPM? Not yet, but the contest is not entirely fair: to make the comparison possible, we must work under assumptions designed specifically for IPMs. Nevertheless, we show that, even in this setting, \eqref{eq:proxMD} performs only slightly worse.

Below, we provide an idealized and simplified comparison in which all subproblems are solved exactly. This assessment should be understood as a proof of concept, while more detailed analyses, closer to those available for short-step interior-point methods, are left for future work.

\subsection{Reminders on Newton's method for self-concordant functions}

We briefly recall a standard complexity bound for damped Newton's method applied to a convex generalized self-concordant function $g:\R^d\to \R\cup\{+\infty\}$.
Consider
\[
\min_{y\in \R^d} g(y)
\]
and assume that $g$ is proper, convex, bounded below, and attains its minimum at some point $y_\star\in \inter(\dom g)$.
We further assume that $g$ is $M_g$-self-concordant on $\inter(\dom g)$ in the sense that
\[
|D^3 g(y)[u,u,u]|
\le
2 M_g \bigl(D^2 g(y)[u,u]\bigr)^{3/2}
\qquad
\forall y\in \inter(\dom g),\ \forall u\in \R^d,
\]
for some constant $M_g>0$.
Given $y\in \inter(\dom g)$, the Newton direction is $d(y)\coloneqq -[\nabla^2 g(y)]^{-1}\nabla g(y)$ and the classical Newton iteration is $y_+=y+\gamma\,d(y)$ for some $\gamma>0$.
Standard analyses are stated in terms of convergence of the Newton decrement \[\lambda_g(y)\coloneqq \sqrt{\ip{\nabla g(y)}{[\nabla^2 g(y)]^{-1}\nabla g(y)}}=\|\nabla g(y)\|_{[\nabla^2 g(y)]^{-1}},\]
and yield global convergence results for various stepsize strategies, such as using a backtracking procedure or damped Newton steps with $\gamma=[1+M_g\lambda_g(y)]^{-1}$ \cite{nemirovski2008interior,nesterov2018lectures}.
Then, standard analyses predict that at most
\begin{equation}\label{eq:global_Newton}
        AM^2_g (g(y_0)-g_\star) + B \log\log\frac1\epsilon 
\end{equation}
iterations are required by Newton's method to find some $\hat{y}$ such that $\lambda_g(\hat y)\leq \epsilon$; see, e.g.,~\cite[Eq.~(5.2.11)]{nesterov2018lectures}.
In~\eqref{eq:global_Newton}, the first term captures the cost of globalization and depends on the initial objective gap, while the second term corresponds to the final local quadratic phase where $\gamma\approx 1$.
The precise values of factors $A$ and $B$ are irrelevant for our purposes, as we only rely on the structure of this bound for our qualitative analyses below. 

\subsection{Total number of Newton iterations}
This section outlines the computational complexity analysis for the problem 
\[
\min_{x\in\CC} f(x),
\]
where we assume that $f$ is $M_f$-self-concordant on $\inter(\CC)$. We also let $h$ be $1$-self-concordant on $\inter(\CC)$ and the subproblem $x_{k}=\argmin_x\{ G_k(x)\}$ be given by $G_k(x)=\alpha_k f(x)+D_h(x,x_{k-1})$. It follows from~\cite[Thm~5.1.1]{nesterov2018lectures} and $\alpha_{k+1}\geq \alpha_k$, that $G_k$ is $M_k$-self-concordant with $M_k\leq \bar{M}\coloneqq \max\left\{\frac{M_f}{\sqrt{\alpha_1}},1\right\}$.

\paragraph{Outer iteration count.}
We aim to find $\hat{x}\in\CC$ such that $f(\hat{x})-f_\star\leq \delta$ for some prescribed accuracy $\delta>0$.
For completeness, we recall $A_k \coloneqq \sum_{i=1}^{k}\alpha_i$ and define
\begin{equation}\label{eq:Dxzeroprox}
        \Dxzeroprox \coloneqq 2\nu + \ip{\nabla h(x_0)}{ x_0-x_\star}
\end{equation}
for the initialization-dependent constant.
Similar to traditional IPMs, we choose for $k\geq 1$ some geometrically growing stepsizes $\alpha_{k+1}=\mu A_k$ with $\mu>0$ for some $A_1=\alpha_1>0$.
Using Corollary~\ref{cor:proxMD_exponentialstepsize}, one thus obtains
\begin{equation}\label{eq:prox_MD_geometric}
        f(x_k)-f_\star
        \leq
        \frac{ \Dxzeroprox + (k-1)\frac \nu 4\log(1+\mu)}{\alpha_1(1+\mu)^{k-1}}.
\end{equation}
Hence, the $\delta$-accuracy requirement is satisfied after at most
$ k = O\left(\frac{\log(\nu/\delta)}{\log(1+\mu)}\right)$ outer iterations.

\paragraph{Inner iteration count.}
We bound the number of Newton steps necessary to solve the inner optimization problem given the initial estimate $x_{k-1}$.
We assume that $x_{k-1}$ minimizes $G_{k-1}$ exactly.
The bound~\eqref{eq:global_Newton} gives
$N_k \leq A\bar M^2\bigl(G_k(x_{k-1})-G_k(x_k)\bigr)+C$
where $N_k$ is the number of inner Newton iterations to solve the $k$-th subproblem,
and $C=B\log\log\frac1\epsilon$ is treated as a constant.
We proceed as follows:
\begin{align*}
G_k(x_{k-1})-G_k(x_k)
={}&
\alpha_k\bigl(f(x_{k-1})-f(x_k)\bigr) - D_h(x_k,x_{k-1})
\\
\leq{}&
\alpha_k\bigl(f(x_{k-1})-f(x_k)\bigr)
\\
\leq{}&
\alpha_k\bigl(f(x_{k-1})-f_\star\bigr)
\\
\leq{}&
\mu \left(\Dxzeroprox + (k-2)\frac \nu 4\log(1+\mu)\right)
,
\end{align*}
where the last inequality follows from~\eqref{eq:prox_MD_geometric} and $\alpha_k = \mu A_{k-1} = \mu \alpha_1 (1+\mu)^{k-2}$.
Thus, for $k\geq 2$ we reach
\[
N_k \leq A\bar{M}^2 \mu \left(\Dxzeroprox + \frac{\nu}{4} \log(1+\mu)(k-2)\right) + C.
\]
In analogous way, for $k=1$ we have $N_1 \leq A\bar{M}^2 \alpha_1 \bigl(f(x_0)-f_\star\bigr) + C$.

\paragraph{Global Newton iteration count.}
The total number of Newton iterations is therefore upper-bounded by the sum of the previous estimates:
\begin{align*}
        \sum_{i=1}^k N_i
        \leq{}&
        A\bar{M}^2 \alpha_1 \bigl(f(x_0)-f_\star\bigr) + C k + A \bar{M}^2 \mu \sum_{i=2}^k \left( \Dxzeroprox + \frac{\nu}{4} \log(1+\mu) (i-2) \right)
        \\
        ={}&
        A\bar{M}^2 \alpha_1 \bigl(f(x_0)-f_\star\bigr) + C k + A \bar{M}^2 \mu \left( \Dxzeroprox k + \frac{\nu}{4} \log(1+\mu) \left(\frac{k(k+1)}{2}-2k\right) \right)
        .
\end{align*}
For fixed initial point $x_0$ and stepsize $\alpha_1$, the leading term in $k$ is $O\left( \mu \nu \log(1+\mu) k^2 \right)$.
From the outer iteration count estimate $k = O\left( \frac{\log(\nu/\delta)}{\log(1+\mu)} \right)$, the number of Newton iterations scales as $O\left( \frac{\mu \nu}{\log(1+\mu)} \log^2(\nu/\delta) \right)$.
Thus, fixing $\mu>0$, we obtain the complexity bound $O\left( \nu\log^2\tfrac{1}{\delta} \right)$.
For comparison, a similar analysis of IPMs yields the classical bound \(O(\nu \log \tfrac{1}{\delta})\).
When the objective is linear, one may also compare with short-step IPMs, which achieve \(O(\sqrt{\nu}\log \tfrac{1}{\delta})\); we note again that, for linear objectives,~\eqref{eq:proxMD} and~\eqref{eq:IPM} are equivalent up to a stepsize adjustment.
For nonlinear objectives, the standard IPM approach is to move the objective into the constraints, for instance through an epigraph reformulation, so that the nonlinear objective is instead handled through the barrier.

\section{Remarks and outlook}

Technically, this work establishes $\log(k)/k$-type bounds under exp-concavity of the mirror function.
While earlier versions of the proof were obtained through analogies with known techniques (inspired by gradient accumulation such as in~\cite{teboulle2023elementary}), the final version presented in Section~\ref{sec:prox_bregman} was refined using performance estimation problems for exp-concave functions, which also allowed us to construct examples on which the main bounds are attained.

\medskip

Potential ways to improve the bounds include a better model for log-barrier
functionals beyond assuming that they are ``just'' exp-concave, such as
log-homogeneity of the barriers, used in the study of interior-point methods.
Beyond this, it is known that accelerated convergence rates are out of reach for general mirror descent methods~\cite{dragomir2022optimal}, but it remains unclear how much this statement resists to stronger assumptions on the distance-generating function~\cite{dragomir2021bregman}.
Other directions include exploring approximate solvers (as in the hybrid
proximal extragradient
framework~\cite{monteiro2010complexity,monteiro2013accelerated}) and
investigating possible numerical advantages of proximal mirror methods combined with interior-point strategies, possibly for specific problem structures.
In particular, interior-point strategies require exponentially growing stepsizes, typically driving the corresponding subproblem's to strong ill-conditioning  while proximal mirror schemes (updating the prox-center) allow for bounded but potentially large stepsizes.

\small
\subsection*{Acknowledgments}
This research is supported by the European Union (ERC grant CASPER 101162889), by the \emph{Agence Nationale de la Recherche} through the \emph{France 2030} program (grant reference ANR-23-IACL-0008, \emph{PR[AI]RIE-PSAI}), and by the Austrian Science Fund (FWF 10.55776/STA223).
Views and opinions expressed are those of the authors only and do not necessarily reflect those of the funding agencies or granting authorities, which cannot be held responsible for them.

\phantomsection
\addcontentsline{toc}{section}{References}%
\bibliographystyle{habbrv}
\bibliography{bib}

\end{document}